\documentclass[twoside]{article}
\usepackage{amsmath,amsthm, amsfonts, amssymb}
 \usepackage{pgf,tikz}
\usepackage{pgfplots}
\usepackage{amsmath,amsthm, amsfonts, amssymb}
\usepackage{enumerate}
\usepackage{geometry}
\usepackage{secdot}
\usepackage{sectsty}
\allsectionsfont{\normalfont}

\theoremstyle{plain}
\newtheorem{thm}{Theorem}[section]
\newtheorem{prop}[thm]{Proposition}
\newtheorem{cor}[thm]{Corollary}
\newtheorem{lem}[thm]{Lemma}
\theoremstyle{definition}
\newtheorem{defn}[thm]{Definition}
\newtheorem{exa}[thm]{Example}

\newtheorem*{question*}{Question}

\newtheorem{rem}[thm]{Remark}
\usepackage[linktocpage, bookmarks]{hyperref}

\usepackage{fancyhdr}
\usepackage{lipsum}

\author{First Author \and Second Author}
\title{CELLULAR AUTOMATA OVER RACKS}

\makeatletter
\newcommand\Author{NAQEEB UR REHMAN, AND MUHAMMAD KHURAM SHAHZAD}
\let\Title\@title
\makeatother

\begin{document}
	
	\title{\textsc{Cellular Automata on Racks}}
	
	\author{\textsc{Naqeeb ur Rehman, and Muhammad Khuram Shahzad}}
	\date{}
	\maketitle
	
\paragraph{ABSTRACT.}In this paper we initiate the study of cellular automata on racks. A rack $R$ is a set with a self-distributive binary operation. The rack $R$ acts on the set $A^R$ of configurations from $R$ to a set $A$. We define the cellular automaton on a rack $R$ as a continuous self-mapping of $A^R$ defined from a system of local rules. The cellular automata on racks do not commute with the rack action. However, under certain conditions, the cellular automata on racks do commute with the rack action. We study the equivariant cellular automta (which commute with the rack action) on racks and prove several properties of these cellular automata including the analog of Curtis-Hedlund’s theorem for cellular automata on groups.

\begin{center}
	\section{\textsc{Introduction}}
\end{center}
\paragraph{} Cellular automata were introduced by John von Neumann
 as discrete dynamical systems in order to describe theoretical models of self-reproducing machines (see \cite{13}). A cellular automaton is a map defined on a set of configurations. A configuration is a map from a set called the universe into another set called the alphabet. The elements of the universe are called cells and the elements of the alphabet are called states.  Originally, the cells were the squares of an infinite $2-$dimensional checker board. Later it had been extended to a higher dimensional board. In modern cellular automaton theory, the cell structure is provided by any group $G$. The left multiplication in $G$ induces a natural action of $G$ on the set of configurations which is called the $G-$shift and all cellular automata commute with the shift i.e. it is $G-$equivariant. We refer the reader to \cite{3} for more details.
  \paragraph{} In this paper we initiate the study of cellular automata on a self-distributive algebraic structure called rack. The binary operation of a rack $R$ is like the conjugation in a group $G$. Racks and subclasses of racks were studied in geometry and knot theory under various names such as kei in \cite{14}, quandle in \cite{8}, distributive groupoid in \cite{11}, and automorphic set in \cite{2}. We refer the reader to \cite{4} for more details. Section 2 recalls the definitions and examples of racks. Section 3 first defines the configuration set on the rack $R$ with a shift action of $R$. It also defines what is a cellular automaton over the configuration set on a rack $R$ and proves its basic properties. The cellular automata on racks are in general non-equivariant under the shift action of a rack. Section 3 also defines the equivariant cellular automata on racks and their memory sets. This class of cellular automata is the one which have the most similarity with the ones on groups. Section 4 investigates the continuous cellular automata on racks with respect to the prodiscrete topology on  the configuration set on a rack. In particular, we prove a characterization of equivariant cellular automata on racks which is an analogue of Curtis-Hedlund's theorem  for cellular automata on groups. Section 5 studies the composition of cellular automata on racks. In particular, it proves that the set of cellular automata on a rack forms a closed self-distributive algebraic structure under the composition of cellular automata.

 \newpage
\begin{center}
\section{\textsc{Preliminaries on Racks}}

\end{center}

\paragraph{} We begin with the definition and examples of rack. We refer the reader to \cite{1}, \cite{4} for more details.

\begin{defn}
	A \emph{rack} is a pair $(R, \rhd)$, where $R$ is a non-empty set and
	$\rhd:R\times R \longrightarrow R$ is a binary operation such that
	\begin{description}
		\item[\;\;\; (R1)] the map $\phi_r : R \longrightarrow R$, defined by $\phi_r(s)=r\rhd s$, is bijective for all $r, s\in R$,
		\item[\;\;\; (R2)] $r \rhd (s \rhd t) = (r \rhd s) \rhd (r \rhd t)$ for all $r, s, t \in R$ (i.e. $\rhd$ is self-distributive).
	\end{description}
\end{defn}
\paragraph{} Since the map $\phi_r : R \longrightarrow R$, defined by $\phi_r(s)=r\rhd s$, is bijective for all $r\in R$, therefore the inverse of $\phi_r$ defines a second binary operation $\rhd^{-1}$ on $R$ such that the map $\phi^{-1}_r : R \longrightarrow R$, defined by $\phi^{-1}_r(s)=r\rhd^{-1} s$, is also bijective, and $\rhd^{-1}$ is left-self distributive. By the self-distributivity of $\rhd$ it follows that the map $\phi_r$ is a rack automorphism which is called the rack inner automorphism. The set of inner automorphisms $\phi_r$ for all $r\in R$ generates a group which is called the \textit{inner group} of a rack $R$. By the self-distributivity of $\rhd$ it follows that, for all $r_{1},r_{2}\in R$,

\begin{center}
$\phi_{r_{1}\rhd r_{2}}= \phi_{r_{1}}\circ\phi_{r_{2}}\circ\phi^{-1}_{r_{1}}$, 
\end{center}
where $\circ$ denotes the composition of functions.
\paragraph{} A \textit{quandle} is a rack $(R, \rhd)$ where every element is idempotent, that is, for every $r\in R$, $r\rhd r=r$. A \textit{subrack} of a rack $(R, \rhd)$ is  a subset $S\subset R$ such $(S, \rhd)$ is also a rack. A \textit{shelf} is a pair $(S, \rhd)$, where $S$ is a non-empty set and $\rhd:S\times S \longrightarrow S$ is a self-distributive binary operation. Note that every rack is a shelf.

\begin{exa}
	Any non-empty set $R$ with operations  $\rhd$ and $\rhd^{-1}$ defined by $r\rhd s = s$, $r\rhd^{-1} s = s$, for all $r,s\in R$, is a quandle called the \emph{trivial quandle}.
\end{exa}
\begin{exa}Let $\mathbb{Z}_n =\{ 0, 1, 2, 3, ... n-1\}$ be the group of residue classes of $\mathbb{Z} \pmod n$. Then $\mathbb{Z}_n$ a is rack with operations $\rhd$ and $\rhd^{-1}$ defined  by
	\begin{center}
		$r\rhd s = s+1$, $r\rhd^{-1} s = s-1$,
	\end{center}
	for any $r,s\in \mathbb{Z}_n$. This is called  the  \emph{cyclic rack}. Note that the cyclic rack is not a quandle because $r\rhd r = r+1 \neq r$.

\end{exa}

\begin{exa}
	Let $G$  be a group. Consider two binary operations $\rhd$ and $\rhd^{-1}$ on $G$ defined by:
	\begin{center}
		$r \rhd s := rsr^{-1} , r \rhd^{-1} s := r^{-1}sr$.
	\end{center}
Then the group $G$ endowed with these two operations is a quandle which is  called the \emph{conjugation quandle} on $G$, denoted as $Conj(G)$.
	
\end{exa}

\begin{exa}
	Let $G$  be a group. Consider two binary operations $\rhd$ and $\rhd^{-1}$ on $G$ defined by:
	\begin{center}
		$r \rhd s := rs^{-1}r , r \rhd^{-1} s := rs^{-1}r$.
	\end{center}
Then the group $G$ endowed with these two operations is a quandle which is  called the \emph{core quandle} on $G$, denoted as $Core(G)$.
	
\end{exa}

\begin{exa}
	The \emph{dihedral quandle} $\mathbb{D}_n$ of order $n$ is defined on $\mathbb{Z}_n=\{0,1,...,n-1\}$ by
	\begin{center}
		$r \rhd s = r \rhd^{-1} s = 2r-s \pmod n$,
	\end{center}
	for all $r, s \in \mathbb{Z}_n$.
\end{exa}
\begin{exa}
		Let $A$ be an abelian group and $Aut(A)$ is the group of automorphisms of $A$. Let $g \in Aut(A)$ and $1=id_A$. Then we have a quandle structure on $A$ given by
		\begin{center}
			$r \rhd s = (1-g)(r) + g(s)$,
			\\$r \rhd^{-1} s = (1-g^{-1})(r) + g^{-1}(s)$,
		\end{center}
		for all $r,s \in A$. This quandle is called the \emph{affine quandle or Alexander quandl} associated to the pair ($A, g)$ and is denoted by \emph{Aff}$(A, g)$.
		\paragraph{} Let $\mathbb{F}_q$ be a finite field, where $q$ is a power of a prime number $p$, and $\alpha \in \mathbb{F}_q \setminus \{0\}$. Then we write \emph{Aff}$(\mathbb{F}_q, \alpha)$ or simply \emph{Aff}$(q, \alpha)$ for the affine quandle \emph{Aff}$(A, g)$, where $A = \mathbb{F}_q$ and $g:\mathbb{F}_q\rightarrow \mathbb{F}_q$ is the automorphism given by $g(s) = \alpha s$ for all $ \in\mathbb{F}_q $. It is easy to see that for $\alpha = 1$,  the affine quandle \emph{Aff}$(q, \alpha)$ is trivial, and  for $\alpha = -1$, the affine quandle \emph{Aff}$(p, -1)$ is the dihedral quandle $\mathbb{D}_p$.
\end{exa}

\begin{rem}
Observe that if $(R, \rhd)$ is a finite rack then any closed subset $S$ of $R$ is a subrack, because if $S$ is closed under $\rhd$, then the restriction $\phi_r|_S : R \rightarrow R$ is bijective for every $r\in S$, and the self-distributivity of $\rhd$ in $S$ is inherited from $R$. However, if $R$ is an infinite rack then a closed subset $S$ of $R$ is not always a subrack (see \cite{9}). We give an example of a closed subset of an infinite rack which is not a subrack.
	
\end{rem}

	\begin{exa} Let $\mathbb{Q}=\{\frac{p}{q} : p,q\in\mathbb{Z},  \text{and}, q \neq 0\}$. Then $\mathbb{Q}$  is a rack under the binary operations $\rhd$ and $\rhd^{-1}$ defined by
		\begin{center}
			$r\rhd s =  \frac{(k-1)r +s}{k}$, $r\rhd^{-1} s =  (1-k)r +ks$,
		\end{center}
for all $r, s\in \mathbb{Q}$, where $k$ is an integer with $k\geq1$. For $k = 2$,  $r\rhd s = \frac{r +s}{2}$.
Now take the set $P$ of positive rationals. Since the average of two positive numbers is positive, therefore $P$ is closed subset of $\mathbb{Q}$ under $\rhd$. However, $P$ is not a subrack of $\mathbb{Q}$  under the operation $\rhd$ because there is no positive rational $s\in \mathbb{Q}$ that satisfies the equation $3\rhd s =\frac{1}{2}$.
		
	\end{exa}

	\begin{defn}
		Let $(R, \rhd)$ be a rack and $X$ be a set. A \textit{rack action} of $R$ on $X$ is a map $\cdot:R\times X\rightarrow X$ such that
		
		\begin{description}
			\item[\;\;\; (1)] the map $\varphi_r : X \longrightarrow X$, defined by $\varphi_r(x)=r\cdot x$, is bijective for all $r\in R$ and $x\in X$,
			\item[\;\;\; (2)] $	r_{1}\cdot(r_{2}\cdot x)=(r_{1}\rhd r_{2})\cdot(r_{1}\cdot x)$, for all $r_{1},r_{2}\in R$ and $x\in X$.
		\end{description}
\end{defn}
A \emph{rack set} or \emph{R-set} is a set $X$ with a rack action by a rack $R$. Note that any rack $(R, \rhd)$ acts on itself via its binary operation $\rhd$.

	\begin{exa}
		Let $R$ be a rack and let $X = \{x_1, x_2, ..., x_n\}$ be a set of cardinality $n$. Then for any
		permutation $\sigma \in S_n$ the rack action is given by $r \cdot x_i = x_{\sigma(i)}$. Indeed, for all $r_1, r_2\in R$,
		\begin{center}
		$r_1 \cdot (r_2 \cdot x) = r_1 \cdot (x_{\sigma(i)})=x_{\sigma^2(i)}= (r_1\rhd r_2) \cdot ( x_{\sigma(i)}) = (r_1 \rhd r_2) \cdot (r_1  \cdot x)$.
		\end{center}
		
	\end{exa}
	
	\begin{exa}
Let $G$ acts on a set $X$. Then this group action defines a rack action of the conjugation rack $Conj(G)$. Indeed, for all $r_1, r_2\in G$,
\begin{center}
$r_1 \cdot (r_2 \cdot x) = (r_1r_2) \cdot x)=  (r_1r_2r_1^{-1}r_1) \cdot x) =(r_1r_2r_1^{-1})(r_1(x)) = (r_1 \rhd r_2)\cdot (r_1 \cdot x)$.
\end{center}
	\end{exa}
	
	\begin{defn}
		Let $R$ be a rack and $X, Y$ be two $R-$sets. Then a map $f: X\longrightarrow Y$ is said to be \emph{$R-$equivariant} if $f(rx)=rf(x)$ for all $x\in X$ and $r\in R$.
		
	\end{defn}
	
	\begin{defn}
		Let $R$ be a rack, $X$ be an $R$-set and $x\in X$. Then the \textit{stabilizer} of $x$ is defined as
		\begin{equation*}
		Stab(x):=\{r\in R:r \cdot x=x\}\subset R.
		\end{equation*}
	\end{defn}
	
	\begin{lem}
		Let $R$ be a rack, $X$ be a rack set. Then $Stab(x)$ is a shelf in $R$ for all $x\in X$.
	\end{lem}
	\begin{proof}
		Let $r_{1},r_{2}\in Stab(x)$. Then $r_{1} \cdot x=x$ and $r_{2} \cdot x=x$. Therefore, $(r_{1}\rhd r_{2}) \cdot x  =(r_{1}\rhd r_{2})\cdot (r_{1}\cdot x)= r_{1}\cdot(r_{2}\cdot x)
	 = r_{1}\cdot x= x$. Hence $ r_{1}\rhd r_{2}\in $ $Stab(x)$, that is, $Stab(x)$ is a shelf in $R$.

	\end{proof}
	
\paragraph{\textbf{The Configuration Space on a Rack with a Shift Action.}}

	Let $R$ be a rack and $A$ be a set. Consider the set $A^R$ consisting of all maps from $R$ to $A$:
	\begin{center}
		$A^R=\prod\limits_{r\in R}A:=\{x:R\longrightarrow A\}$.
	\end{center}
	The set $A$ is called the \emph{alphabet}. The elements of $A$ are called the \emph{letters}, or the \emph{states}, or the \emph{symbols}, or the \emph{colors}. The rack $R$ is called the \emph{universe}. The set $A^R$ is called the \emph{set of configurations} or simply the \emph{configuration space} on a rack $R$.
	
	 Given an element $r \in R$ and a configuration $x \in A^R$, we define the configuration $r \cdot x \in A^R$ by
	\begin{center}
		$r\cdot x:=x\circ \phi^{-1}_r$,
	\end{center}
	where $\circ$ denotes the composition of maps. We prove that this defines the rack action of $R$ on the configuration space $A^R:=\{x:R\longrightarrow A\}$.
\begin{prop}
Let $R$ be a rack and $A$ be an alphabet. Then $R$ acts on the configuration space $A^R$ as $r\cdot x=x\circ \phi^{-1}_r$ for all $r\in R$ and $x\in A^R$.
\end{prop}

\begin{proof}
	The map $\varphi_r : A^R \longrightarrow A^R$, defined by $\varphi_r(x)=r\cdot x=x\circ \phi^{-1}_r,$ is bijective since there exists a map $\varphi^{-1}_r : A^R \longrightarrow A^R$, defined by $\varphi^{-1}_r(x)=x\circ \phi_r$, such that $\varphi \circ \varphi^{-1}_r(x)=\varphi(x\circ \phi_r)=x\circ \phi_r\circ \phi^{-1}_r=x$ and $\varphi^{-1} \circ \varphi_r(x)=\varphi^{-1}(x\circ \phi^{-1}_r)=x\circ \phi^{-1}_r\circ \phi_r=x$. Moreover, since $\phi_{r_{1}\rhd r_{2}}= \phi_{r_{1}}\circ\phi_{r_{2}}\circ\phi^{-1}_{r_{1}}$ for all $r_1, r_2 \in R$, $r_1 \cdot (r_2 \cdot x)  = r_1\cdot (x\circ \phi^{-1}_{r_2})
	= (x\circ \phi^{-1}_{r_2})\circ \phi^{-1}_{r_1}
	= x\circ (\phi^{-1}_{r_2}\circ \phi^{-1}_{r_1})
	= x\circ (\phi_{r_1}\circ \phi_{r_2})^{-1}
	= x\circ (\phi_{r_1\rhd r_2}\circ \phi_{r_1})^{-1}
	= x\circ (\phi^{-1}_{r_1}\circ \phi^{-1}_{r_1\rhd r_2})
	= (x\circ \phi^{-1}_{r_1})\circ \phi^{-1}_{r_1\rhd r_2}
	= (r_1\cdot x)\circ \phi^{-1}_{r_1\rhd r_2}
	= (r_1 \rhd r_2) \cdot (r_1 \cdot x).$
	
\end{proof}

\begin{exa}
For a trivial rack $R$, $r\cdot x=x$ for all $r \in R$ and $x\in A^R$ since $r\rhd ^{-1}s=s$ and $(r\cdot x)(s)=(x\circ \phi^{-1}_r)(s)=x(r\rhd ^{-1}s)=x(s)$ for all $s\in R$. Therefore, for a trivial rack $R$, $Stab(x)=R$ for all $x\in A^R$. For a conjugation rack $Conj(G)$ on a group $G$ with identity element $1_G$, $Stab(x)\neq \emptyset$ because in $Conj(G)$, $1_G \rhd^{-1} s = {1^{-1}_G}s1_G=s$ for all $s\in G$ and therefore $1_G \in Stab(x)$ for all $x\in A^R$.
	
\end{exa}
	
	\begin{exa}
		Let $S_{3}$ denotes the symmetric group on a set with three elements.  Take a subrack $R=\{r_{1}, r_{2}, r_{3}\}$ of  $Conj(S_{3})$ with $r_{1}=(23)$, $r_{2}=(13)$ and $r_{3}=(12)$. Let $A=\{0,1\}$, and $ A^{R} = \{x_{ijk}:i,j,k\in \{0,1\}\}$, where $x_{ijk}:R\rightarrow A$ are the configurations defined by
		 $x_{ijk}(r_{1})=i$, $x_{ijk}(r_{2})=j$, $x_{ijk}(r_{3})=k$.  The stabilizers of elements of $A^{R}$ are
		$Stab(x_{000})=R$, $Stab(x_{111})=R$, $Stab(x_{100})=\{(23)\}$, $Stab(x_{010})=\{(13)\}$, $Stab(x_{001})=\{(12)\}$, $Stab(x_{110})=\{(12)\}$,
		$Stab(x_{101})=\{(13)\}$, $Stab(x_{011})=\{(23)\}$.
	\end{exa}
	\begin{exa}
		Consider the rack $(\mathbb{Z}, \rhd)$ on the set of integers with operation $r\rhd s=2r-s$ and $r\rhd^{-1}s=2r-s$ for all  $r,s\in \mathbb{Z}$. Take $A=\{0,1\}$. Let $x\in A^{R}$ defined by
		
		\[ x(s) = \left\{
		\begin{array}{l l}
		1 & \text{if}\ s\in 2\mathbb{Z}\;\\~\\ 0 & \text{if}\ s\in 1+2\mathbb{Z},\;

		\end{array} \right.\]
		for all $s\in \mathbb{Z}$. Then for all $r\in \mathbb{Z}$, $r\cdot x=x\circ \phi_{r}^{-1}=x$, therefore $Stab(x)=\mathbb{Z}$.
		
	\end{exa}

\begin{prop}
	Let $R$ be a rack and $A$ be an alphabet. Let $r\in R $ and $x\in{A^{R}}$ be a configuration. Then for all $s\in R$, $x(s)=x(r\rhd^{-1}s)$ if and only if $r\in Stab(x)$.
\end{prop}
\begin{proof}
	Suppose $x(s)=x(r\rhd^{-1} s)$. Then for all $s\in R$,
	$x(s)=x(\phi_{r}^{-1}(s))=(x\circ \phi_{r}^{-1})(s)$. This implies that $x=x\circ \phi_{r}^{-1}=r \cdot x$. Therefore, $r\in Stab(x)$.
\end{proof}

	\begin{center}
	\section{\textsc{Cellular Automata on Racks}}
\end{center}

\begin{defn}
Let $ R $ be rack and $ A $ be a set. A cellular automaton on $ A^R $ is a map $ \tau: A^{R}\longrightarrow A^{R} $ satisfying the following property: there exist a finite subset $ M\subset R $ and a map $ \mu:A^{M}\longrightarrow A $ such that
	\begin{equation} \label{1}
	\tau(x)(r)= \mu(r\cdot x|_{M}),
	\end{equation}
	for all $ x\in A^{R} $ and $ r\in R $, where $(r\cdot x|_{M})$ denotes the restriction of the configuration $r\cdot x$ to $M$. Such a set $ M $ is called a \textit{ memory set} and $ \mu $ is called a \textit{local defining} map for $ \tau $.
\end{defn}
\begin{rem}
(a)  The formula (\ref{1}) says that the value of the configuration $ \tau(x) $ at an element $ r\in R $ is the value taken by the local defining map $ \mu $ at the pattern $p: M \longrightarrow A$ obtained by restricting to the memory set $ M $ the shifted configuration $r\cdot x$.
(b) The equality (\ref{1}) may also be written as $\tau(x)(r)= \mu((x\circ \phi_{r}^{-1})|_{M})$, since $r \cdot x= x\circ \phi_{r}^{-1}$.
Moreover, if $r\in Stab(x)$, then formula (\ref{1}) becomes
$\tau(x)(r)=\mu(x|_M)$.
(c)	Since $(r\cdot x)(m)=(x\circ \phi_{r}^{-1})(m) =x(\phi_{r}^{-1}(m)) =x(r\rhd^{-1}m)$ for all $m\in M$, therefore $\tau(x)(r)$ depends only on the restriction of $x$ to $r\rhd^{-1}M=\{r\rhd^{-1}m : m\in M\}$.

\end{rem}	

\begin{exa}
	Let $\mathbb{Z}_{3}=\{0, 1, 2\}$ be a dihedral rack under the operations $r\rhd s=r\rhd^{-1} s=2r-s~(mod~3)$, and $A=\{0,1\}$ be any set. Define a map
	$\tau:A^{\mathbb{Z}_{3}}\rightarrow A^{\mathbb{Z}_{3}}$ by:
	\begin{center}
		$	\tau(x_{ijk})  =
		\begin{cases}
		x_{ijk} ~~~~~~~ \text{if}~~
		i=j\\
		x_{jik} ~~~~~~~\text{if}~~
		i\neq j\\
		\end{cases}$
	\end{center}
	for all $x_{ijk}=(i, j, k)\in A^{3}$, with $ i,j,k\in \{0,1\} $. Then $\tau$ is a cellular automaton on $A^{\mathbb{Z}_{3}}$ with memory set $ M=\{0\} $ and the local defining map $\mu:A^{M}\rightarrow A $ defined by:
	\begin{center}
		$	\mu(y_{\alpha})  =
		\begin{cases}
		0 ~~~~~~~ \text{if}~~
		\alpha=0\\
		1 ~~~~~~~\text{if}~~
		\alpha=1,\\
		\end{cases}$
		
	\end{center}
for all $y_{\alpha}\in A^{M}$, where $\alpha\in \{0, 1\}$.

\end{exa}

\begin{exa}
	\textbf{The Majority action cellular automaton on a rack.}
	Let $G$ be a group with identity element $1_G$ and $ M $ be a finite subset of $ G $. Consider the conjugation rack $Conj(G)$. Take $ A=\{0,1\} $ and consider the map $ \tau: A^{R}\longrightarrow A^{R} $ defined by
	
	\begin{center}
		
		$	\tau(x)(r)  =
		\begin{cases}
		1 ~~~~~~~ \text{if}~~
		\sum_{m\in M}x(r\rhd m)>\dfrac{|M|}{2}\\
		0 ~~~~~~~\text{if}~~
		\sum_{m\in M}x(r\rhd m)<\dfrac{|M|}{2}\\
		x(r) ~~~\text{if}~~
		\sum_{m\in M}x(r\rhd m)=\dfrac{|M|}{2},\\
		\end{cases}$
		
	\end{center}
for all $ x\in A^{R} $. Then $ \tau $ is a cellular automaton over $Conj(G)$ with memory set $ M\cup \{1_G\} $ and local defining map $ \mu:A^{M\cup \{1_G\}} \longrightarrow A $ given by
	\begin{center}
		
		$	\mu(y)  =
		\begin{cases}
		1 ~~~~~~~~~ \text{if}~~
		\sum_{m\in M}y(m)>\dfrac{|M|}{2}\\
		0 ~~~~~~~~~\text{if}~~
		\sum_{m\in M}y(m)<\dfrac{|M|}{2}\\
		y(1_G) ~~~~\text{if}~~
		\sum_{m\in M}y(m)=\dfrac{|M|}{2},\\
		\end{cases}$
		
	\end{center}
for all $ y\in A^{M\cup \{1_G\}}$. The cellular automaton $ \tau $ is called the majority action cellular automaton associated with the conjugation rack $Conj(G)$. The local defining map $ \mu$ for the majority action on the affine rack of integers can be described in the following figure;
	
	
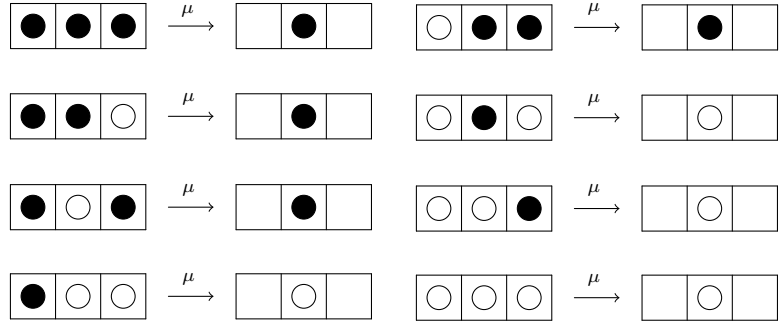
\begin{figure}[hb!]
		\centering
		\begin{tikzpicture}[line cap=round,line join=round,x=.6cm,y=.6cm]
		\clip(-3.061941928207096,-2.095553748369594) rectangle (14.0485894716073377,5.153284350394333);
		\draw [->] (0.5,4.5) -- (1.5,4.5);
		\draw (-3.,5.)-- (-3.,4.);
		\draw (0.,5.)-- (0.,4.);
		\draw (-3.,5.)-- (0.,5.);
		\draw (-3.,4.)-- (0.,4.);
		\draw (-2.,5.)-- (-2.,4.);
		\draw (-1.,5.)-- (-1.,4.);
		\draw (2.,5.)-- (2.,4.);
		\draw (5.,5.)-- (5.,4.);
		\draw (2.,5.)-- (5.,5.);
		\draw (2.,4.)-- (5.,4.);
		\draw (3.,5.)-- (3.,4.);
		\draw (4.,5.)-- (4.,4.);
		\draw [->] (9.495016671418924,4.469822472510882) -- (10.495016671418924,4.469822472510882);
		\draw (5.9950166714189255,4.969822472510882)-- (5.9950166714189255,3.9698224725108817);
		\draw (8.995016671418924,4.969822472510882)-- (8.995016671418924,3.9698224725108817);
		\draw (5.9950166714189255,4.969822472510882)-- (8.995016671418924,4.969822472510882);
		\draw (5.9950166714189255,3.9698224725108817)-- (8.995016671418924,3.9698224725108817);
		\draw (6.995016671418925,4.969822472510882)-- (6.995016671418926,3.9698224725108817);
		\draw (7.995016671418924,4.969822472510882)-- (7.995016671418924,3.9698224725108817);
		\draw (10.995016671418924,4.969822472510882)-- (10.995016671418924,3.9698224725108817);
		\draw (13.995016671418924,4.969822472510882)-- (13.995016671418924,3.9698224725108817);
		\draw (10.995016671418924,4.969822472510882)-- (13.995016671418924,4.969822472510882);
		\draw (10.995016671418924,3.9698224725108817)-- (13.995016671418924,3.9698224725108817);
		\draw (11.995016671418924,4.969822472510882)-- (11.995016671418924,3.9698224725108817);
		\draw (12.995016671418924,4.969822472510882)-- (12.995016671418924,3.9698224725108817);
		\draw [->] (0.49704335349871576,2.4985150726652066) -- (1.4970433534987158,2.4985150726652066);
		\draw (-3.0029566465012842,2.9985150726652066)-- (-3.0029566465012842,1.9985150726652066);
		\draw (-0.0029566465012842436,2.9985150726652066)-- (-0.0029566465012842436,1.9985150726652066);
		\draw (-3.0029566465012842,2.9985150726652066)-- (-0.0029566465012842436,2.9985150726652066);
		\draw (-3.0029566465012842,1.9985150726652066)-- (-0.0029566465012842436,1.9985150726652066);
		\draw (-2.0029566465012842,2.9985150726652066)-- (-2.0029566465012842,1.9985150726652066);
		\draw (-1.0029566465012842,2.9985150726652066)-- (-1.0029566465012842,1.9985150726652066);
		\draw (1.9970433534987158,2.9985150726652066)-- (1.9970433534987158,1.9985150726652066);
		\draw (4.9970433534987135,2.9985150726652066)-- (4.9970433534987135,1.9985150726652066);
		\draw (1.9970433534987158,2.9985150726652066)-- (4.9970433534987135,2.9985150726652066);
		\draw (1.9970433534987158,1.9985150726652066)-- (4.9970433534987135,1.9985150726652066);
		\draw (2.997043353498715,2.9985150726652066)-- (2.9970433534987158,1.9985150726652066);
		\draw (3.9970433534987144,2.9985150726652066)-- (3.9970433534987144,1.9985150726652066);
		\draw [->] (9.492060024917638,2.4683375451760883) -- (10.492060024917638,2.4683375451760883);
		\draw (5.992060024917639,2.9683375451760883)-- (5.992060024917639,1.9683375451760883);
		\draw (8.992060024917638,2.9683375451760883)-- (8.992060024917638,1.9683375451760883);
		\draw (5.992060024917639,2.9683375451760883)-- (8.992060024917638,2.9683375451760883);
		\draw (5.992060024917639,1.9683375451760883)-- (8.992060024917638,1.9683375451760883);
		\draw (6.992060024917638,2.9683375451760883)-- (6.992060024917641,1.9683375451760883);
		\draw (7.992060024917638,2.9683375451760883)-- (7.992060024917638,1.9683375451760883);
		\draw (10.992060024917638,2.9683375451760883)-- (10.992060024917638,1.9683375451760883);
		\draw (13.992060024917638,2.9683375451760883)-- (13.992060024917638,1.9683375451760883);
		\draw (10.992060024917638,2.9683375451760883)-- (13.992060024917638,2.9683375451760883);
		\draw (10.992060024917638,1.9683375451760883)-- (13.992060024917638,1.9683375451760883);
		\draw (11.992060024917638,2.9683375451760883)-- (11.992060024917638,1.9683375451760883);
		\draw (12.992060024917638,2.9683375451760883)-- (12.992060024917638,1.9683375451760883);
		\draw [->] (0.49704335349871576,0.483902925736011) -- (1.4970433534987158,0.483902925736011);
		\draw (-3.0029566465012842,0.9839029257360108)-- (-3.0029566465012842,-0.016097074263989042);
		\draw (-0.0029566465012842436,0.9839029257360108)-- (-0.0029566465012842436,-0.016097074263989042);
		\draw (-3.0029566465012842,0.9839029257360108)-- (-0.0029566465012842436,0.9839029257360108);
		\draw (-3.0029566465012842,-0.016097074263989042)-- (-0.0029566465012842436,-0.016097074263989042);
		\draw (-2.0029566465012842,0.9839029257360108)-- (-2.0029566465012842,-0.016097074263989042);
		\draw (-1.0029566465012842,0.9839029257360108)-- (-1.0029566465012842,-0.016097074263989042);
		\draw (1.9970433534987158,0.9839029257360108)-- (1.9970433534987158,-0.016097074263989042);
		\draw (4.997043353498716,0.9839029257360108)-- (4.997043353498716,-0.016097074263989042);
		\draw (1.9970433534987158,0.9839029257360108)-- (4.997043353498716,0.9839029257360108);
		\draw (1.9970433534987158,-0.016097074263989042)-- (4.997043353498716,-0.016097074263989042);
		\draw (2.9970433534987158,0.9839029257360108)-- (2.9970433534987158,-0.016097074263989042);
		\draw (3.9970433534987153,0.9839029257360108)-- (3.9970433534987153,-0.016097074263989042);
		\draw [->] (9.492060024917636,0.4537253982468927) -- (10.492060024917636,0.4537253982468927);
		\draw (5.992060024917642,0.9537253982468925)-- (5.992060024917642,-0.046274601753107214);
		\draw (8.992060024917636,0.9537253982468925)-- (8.992060024917636,-0.046274601753107214);
		\draw (5.992060024917642,0.9537253982468925)-- (8.992060024917636,0.9537253982468925);
		\draw (5.992060024917642,-0.046274601753107214)-- (8.992060024917636,-0.046274601753107214);
		\draw (6.99206002491764,0.9537253982468925)-- (6.9920600249176434,-0.046274601753107214);
		\draw (7.992060024917638,0.9537253982468925)-- (7.992060024917638,-0.046274601753107214);
		\draw (10.992060024917636,0.9537253982468925)-- (10.992060024917636,-0.046274601753107214);
		\draw (13.992060024917636,0.9537253982468925)-- (13.992060024917636,-0.046274601753107214);
		\draw (10.992060024917636,0.9537253982468925)-- (13.992060024917636,0.9537253982468925);
		\draw (10.992060024917636,-0.046274601753107214)-- (13.992060024917636,-0.046274601753107214);
		\draw (11.992060024917636,0.9537253982468925)-- (11.992060024917636,-0.046274601753107214);
		\draw (12.992060024917636,0.9537253982468925)-- (12.992060024917636,-0.046274601753107214);
		\draw [->] (0.49704335349871576,-1.4930529193814237) -- (1.4970433534987158,-1.4930529193814237);
		\draw (-3.0029566465012842,-0.9930529193814237)-- (-3.0029566465012842,-1.993052919381424);
		\draw (-0.0029566465012842436,-0.9930529193814237)-- (-0.0029566465012842436,-1.993052919381424);
		\draw (-3.0029566465012842,-0.9930529193814237)-- (-0.0029566465012842436,-0.9930529193814237);
		\draw (-3.0029566465012842,-1.993052919381424)-- (-0.0029566465012842436,-1.993052919381424);
		\draw (-2.0029566465012842,-0.9930529193814237)-- (-2.0029566465012842,-1.993052919381424);
		\draw (-1.0029566465012842,-0.9930529193814237)-- (-1.0029566465012842,-1.993052919381424);
		\draw (1.9970433534987158,-0.9930529193814237)-- (1.9970433534987158,-1.993052919381424);
		\draw (4.997043353498716,-0.9930529193814237)-- (4.997043353498716,-1.993052919381424);
		\draw (1.9970433534987158,-0.9930529193814237)-- (4.997043353498716,-0.9930529193814237);
		\draw (1.9970433534987158,-1.993052919381424)-- (4.997043353498716,-1.993052919381424);
		\draw (2.9970433534987158,-0.9930529193814237)-- (2.9970433534987153,-1.993052919381424);
		\draw (3.9970433534987153,-0.9930529193814237)-- (3.9970433534987153,-1.993052919381424);
		\draw [->] (9.492060024917633,-1.523230446870542) -- (10.49206002491764,-1.523230446870542);
		\draw (5.992060024917638,-1.023230446870542)-- (5.992060024917638,-2.023230446870542);
		\draw (8.992060024917633,-1.023230446870542)-- (8.992060024917633,-2.023230446870542);
		\draw (5.992060024917638,-1.023230446870542)-- (8.992060024917633,-1.023230446870542);
		\draw (5.992060024917638,-2.023230446870542)-- (8.992060024917633,-2.023230446870542);
		\draw (6.992060024917636,-1.023230446870542)-- (6.992060024917638,-2.023230446870542);
		\draw (7.992060024917635,-1.023230446870542)-- (7.992060024917635,-2.023230446870542);
		\draw (10.992060024917643,-1.023230446870542)-- (10.992060024917643,-2.023230446870542);
		\draw (13.992060024917643,-1.023230446870542)-- (13.992060024917643,-2.023230446870542);
		\draw (10.992060024917643,-1.023230446870542)-- (13.992060024917643,-1.023230446870542);
		\draw (10.992060024917643,-2.023230446870542)-- (13.992060024917643,-2.023230446870542);
		\draw (11.992060024917643,-1.023230446870542)-- (11.992060024917643,-2.023230446870542);
		\draw (12.992060024917643,-1.023230446870542)-- (12.992060024917643,-2.023230446870542);
		\begin{scriptsize}
		\draw [fill=black] (-2.5,4.5) circle (4.5pt);
		\draw [fill=black] (-1.5,4.5) circle (4.5pt);
		\draw [fill=black] (-0.5,4.5) circle (4.5pt);
		\draw[color=black] (0.9237668183689371,4.833205784994368) node {$\mu$};
		\draw [fill=black] (3.5,4.5) circle (4.5pt);
		\draw [color=black] (6.4950166714189255,4.469822472510882) circle (4.5pt);
		\draw [fill=black] (7.495016671418926,4.469822472510882) circle (4.5pt);
		\draw [fill=black] (8.495016671418924,4.469822472510882) circle (4.5pt);
		\draw[color=black] (9.923622951379738,4.889690237712009) node {$\mu$};
		\draw [fill=black] (12.495016671418924,4.469822472510882) circle (4.5pt);
		\draw [fill=black] (-2.5029566465012842,2.4985150726652066) circle (4.5pt);
		\draw [fill=black] (-1.5029566465012842,2.4985150726652066) circle (4.5pt);
		\draw [color=black] (-0.5029566465012842,2.4985150726652066) circle (4.5pt);
		\draw[color=black] (0.9425949692748175,2.893906241688694) node {$\mu$};
		\draw [fill=black] (3.4970433534987158,2.4985150726652066) circle (4.5pt);
		\draw [color=black] (6.492060024917639,2.4683375451760883) circle (4.5pt);
		\draw [fill=black] (7.492060024917641,2.4683375451760883) circle (4.5pt);
		\draw [color=black] (8.492060024917638,2.4683375451760883) circle (4.5pt);
		\draw[color=black] (9.923622951379738,2.875078090782813) node {$\mu$};
		\draw [color=black] (12.492060024917638,2.4683375451760883) circle (4.5pt);
		\draw [fill=black] (-2.5029566465012842,0.483902925736011) circle (4.5pt);
		\draw [color=black] (-1.5029566465012842,0.483902925736011) circle (4.5pt);
		\draw [fill=black] (-0.5029566465012842,0.483902925736011) circle (4.5pt);
		\draw[color=black] (0.9425949692748175,0.879294094759498) node {$\mu$};
		\draw [fill=black] (3.4970433534987158,0.483902925736011) circle (4.5pt);
		\draw [color=black] (6.492060024917642,0.4537253982468927) circle (4.5pt);
		\draw [color=black] (7.4920600249176434,0.4537253982468927) circle (4.5pt);
		\draw [fill=black] (8.492060024917636,0.4537253982468927) circle (4.5pt);
		\draw[color=black] (9.923622951379738,0.8604659438536177) node {$\mu$};
		\draw [color=black] (12.492060024917636,0.4537253982468927) circle (4.5pt);
		\draw [fill=black] (-2.5029566465012842,-1.4930529193814237) circle (4.5pt);
		\draw [color=black] (-1.5029566465012842,-1.4930529193814237) circle (4.5pt);
		\draw [color=black] (-0.5029566465012842,-1.4930529193814237) circle (4.5pt);
		\draw[color=black] (0.9425949692748175,-1.0976617503579367) node {$\mu$};
		\draw [color=black] (3.4970433534987153,-1.4930529193814237) circle (4.5pt);
		\draw [color=black] (6.492060024917638,-1.523230446870542) circle (4.5pt);
		\draw [color=black] (7.492060024917638,-1.523230446870542) circle (4.5pt);
		\draw [color=black] (8.492060024917633,-1.523230446870542) circle (4.5pt);
		\draw[color=black] (9.923622951379738,-1.1164899012638172) node {$\mu$};
		\draw [color=black] (12.492060024917643,-1.523230446870542) circle (4.5pt);
		\end{scriptsize}
		\end{tikzpicture}
	\end{figure}
	
\end{exa}

\begin{prop}
Let $R$ be a rack and $A$ be an alphabet. Let $\tau:A^{R}\rightarrow A^{R}$ be a cellular automaton on $R$. Then for all $r, s \in R$ and $x\in A^R$,  $\tau(r\cdot x)(s)=\tau(s\cdot x)(s\rhd r)$ and $r\cdot\tau(x)(s)=\tau(x)(r\rhd^{-1}s)$.
\end{prop}

\begin{proof}
From the definition of $\tau$, $\tau(r\cdot x)(s) =\mu(s\cdot(r\cdot x)|_M)=\mu((s\rhd r)\cdot(s\cdot x)|_M)=\tau(s\cdot x)(s\rhd r)$. Also, $r\cdot\tau(x)(s)=\tau(x)(r\rhd^{-1}s)$.
\end{proof}
The above proposition shows that the cellular automaton $\tau$ on a rack $R$ is not an $R$-equivariant map in general. For equivariant cellular automata on racks we need the following definition.

\begin{defn}
Let $R$ be a rack and $A$ be a set. Let $\tau:A^{R}\rightarrow A^{R}$ be a map. Then $Eq(\tau)$ denotes a subset of $R$ defined as
	\begin{center}
		$ Eq(\tau):=\{r\in R: \tau(r\cdot x)=r\cdot\tau(x)~\forall x\in A^{R}\}$.
	\end{center}
\end{defn}

\begin{prop}
Let $R$ be a rack and $A$ be a set. Let $\tau:A^{R}\rightarrow A^{R}$ be a map. Then $Eq(\tau)$ is a shelf in $R$.
\end{prop}

\begin{proof}
Let $r_{1},r_{2}\in Eq(\tau)$. Then $\tau(r_1\cdot x)=r_1\cdot\tau(x), \tau(r_2\cdot x)=r_2\cdot\tau(x)$. By Proposition 2.15., the map $\varphi_r : A^R \longrightarrow A^R$, defined by $\varphi_r(x)=r\cdot x$, is bijective for all $r\in R$ and $x\in A^R$. Therefore, for all $y\in A^R$, $r_1\cdot x=y$. Now  we have $(r_{1}\rhd r_{2})\cdot \tau (y) =(r_{1}\rhd r_{2})\cdot \tau (r_1\cdot x)=(r_{1}\rhd r_{2})\cdot(r_{1}\cdot \tau (x))
=r_{1}\cdot(r_{2}\cdot \tau (x))=r_{1}\cdot \tau (r_{2}\cdot x)=\tau (r_{1}\cdot(r_{2}\cdot  x)=\tau ((r_{1}\rhd r_{2})\cdot (r_1 \cdot x))=\tau ((r_{1}\rhd r_{2})\cdot y)$ for all $y\in A^R$. This shows that $r_{1}\rhd r_{2}\in Eq(\tau)$, and therefore, $Eq(\tau)$ is a shelf in $R$.
\end{proof}

\begin{exa}
For a conjugation rack $Conj(G)$ on a group $G$, $Eq(\tau)$ is non-empty because $1_{G} \cdot x=x$ for all $x\in R$, and therefore, $\tau(1_{G}\cdot x)=1_{G}\cdot\tau(x)$. Hence $1_{G}\in Eq(\tau)$. If $R$ is a trivial rack then $Eq(\tau)=R$ because for all $x\in R$, $r\cdot x=x$, and therefore, $\tau(r\cdot x)=\tau(x)=r\cdot \tau(x)$.
\end{exa}

\begin{defn}
Let $S$ be a shelf in a rack $R$. Then a map
$\tau:A^{R}\rightarrow A^{R}$ is said to be $S-equivariant$ if $S \subset Eq (\tau)$, i.e., for all $s\in S$ and for all $x\in A^R$,
we have $\tau(s\cdot x)=s\cdot \tau(x)$.
\end{defn}

\begin{lem}
Let $R$ be a trivial rack and $A$ be an alphabet. Then every cellular automaton $\tau:A^{R}\rightarrow A^{R}$ on $R$ is $R$-equivariant.
\end{lem}
\begin{proof}
For a trivial rack $R$, $Eq(\tau)=R$. Therefore, every cellular automaton $\tau:A^{R}\rightarrow A^{R}$ on a trivial rack $R$ is $R$-equivariant.
\end{proof}

\paragraph{} For characterization of the equivariance maps $\tau:A^{R}\rightarrow A^{R}$ as a cellular automaton on a rack $R$, we define the set
\begin{center}
$Stab(x, Eq(\tau)):= \{r\in Eq(\tau): r\cdot x=x\}= \{r\in R: r\cdot x=x, r\cdot \tau(x)=\tau(r\cdot x)\}$
\end{center}
for $x\in A^R$.

\begin{prop}
Let $\tau:A^{R}\rightarrow A^{R}$ be a map. Then $Stab(x, Eq(\tau))$ is a shelf in $Eq(\tau)$ and $Stab(x, Eq(\tau)) \subset Stab(\tau(x), Eq(\tau))$. Moreover, $Stab(x, Eq(\tau)) = Stab(\tau(x), Eq(\tau))$ if $\tau$ is injective.
\end{prop}
\begin{proof}
Since $Stab(x)$ and $Eq(\tau)$ are shelves in $R$, the set $Stab(x, Eq(\tau))$ is a shelf in $Eq(\tau)$. Since $r\cdot \tau(x)=\tau (r\cdot x)=\tau (x)$ for all $r\in Stab(x, Eq(\tau))$, $Stab(x, Eq(\tau)) \subset Stab(\tau(x), Eq(\tau))$. Moreover, if $\tau$ is injective then $Stab(x, Eq(\tau)) = Stab(\tau(x), Eq(\tau))$ because $r\in Stab(\tau(x), Eq(\tau))$ implies that $\tau(r\cdot x)=r\cdot \tau(x)=\tau(x)$, and the injectivity of $\tau$ implies that $r\cdot x=x$, that is, $r\in Stab(x, Eq(\tau))$.
\end{proof}
\begin{prop}
	Let $R$ be a rack and $A$ be a set. Consider a function $\tau:A^{R}\rightarrow A^{R}$ and the shelf $S=Stab(x, Eq(\tau))$ in $R$. Let $M$ be a finite subset of $R$ and let $\mu:A^{M}\rightarrow A$ be a map. Then following conditions are equivalent:
	\item{(a)} $\tau_S:A^{S}\rightarrow A^{S}$ is a cellular automata admitting $M$ as a memory set and $\mu$ as the associated local defining map;
	
	\item{(b)} $\tau_S$ is $S$-equivariant and $\tau_S(x)(r)=\mu(x|_{M})$ for all $r\in S$.
\end{prop}
\begin{proof}
	Suppose (a). Then for all $r, r^{\prime}\in S$, we have $\tau_S(r\cdot x)(r^{\prime})=\mu(r^{\prime}\cdot(r\cdot x)|_{M})=\mu(r^{\prime}\cdot x|_{M})=\tau_S(x)(r^{\prime})=(r\cdot\tau_S(x))(r^{\prime})$.
Hence $\tau_S(r\cdot x)=r\cdot\tau_S(x)$ for all $r\in S$, and therefore, $\tau_S$ is $S$-equivariant. Next, since $r\cdot x=x$ for all $r\in S$, $\tau_S(x)(r)=\mu(r\cdot x|_{M})=\mu(x|_{M})$.
\paragraph{} Conversely, suppose (b). Then, by using the definition of $S=Stab(x, Eq(\tau))$, we get
	$\tau_S(x)(r)=\mu(x|_{M})
	=\mu(r\cdot x|_{M})$
for all $x\in A^{S}$ and $r\in S$. Consequently,  $\tau _S $ satisfies (a).
\end{proof}
\paragraph{} Since $S=Stab(x, Eq(\tau))=R$ for a trivial rack $R$ and a map $\tau:A^{R}\rightarrow A^{R}$, the following corollary is straightforward.
\begin{cor}
Let $R$ be a trivial rack, $A$ be a set, and $\tau:A^{R}\rightarrow A^{R}$ be a map.  Let $M$ be a finite subset of $R$. Let $\mu:A^{M}\rightarrow A$ be a local defining map. Then following conditions are equivalent.\\
	\item{(a)} $\tau  $ is a cellular automata admitting $M$ as a memory set and $\mu$ as the associated local defining map:
	
	\item{(b)} $\tau$ is $R$-equivariant and $\tau(x)(r)=\mu(x|_{M})$ for all $r\in R$.
\end{cor}

\paragraph{\textbf{Minimal Memory.}} Let $R$ be a rack and $A$ be a set. Let $\tau:A^{R}\rightarrow A^{R}$ be a cellular automaton and $S=Stab(x, Eq(\tau))$. Let $M$ be memory set for $\tau$ and let $\mu:A^{M}\rightarrow A$ be the associated local defining map. If $M^{\prime}$ is a finite subset of $R$ such that $M\subset M^{\prime}$, then $M^{\prime}$ is also a memory set for $\tau$ and the local defined map for $M^{\prime}$ is the map $\mu^{\prime}:A^{M^{\prime}}\rightarrow A$ given by $\mu^{\prime}=\mu\circ p$, where $ p:A^{M^{\prime}}\rightarrow A^{M}$ is the canonical restriction map. This shows that the memory set of a cellular automaton $\tau:A^{R}\rightarrow A^{R}$ is not unique in general. However, we shall see that every cellular automaton $ \tau_S $ admits a unique memory set of minimal cardinality. For that we first have the following result.

\begin{lem}
	Let $ \tau : A^{R}\longrightarrow A^{R} $ be a cellular automaton and $S=Stab(x, Eq(\tau))$. Let $ M_{1} $ and $ M_{2} $ be memory sets for $ \tau_S $. Then $ M_{1}\cap M_{2} $ is also a memory set for $ \tau_S $.

\end{lem}

\begin{proof}
	
Let $ x\in A^{S}$. We show that for all $r\in S$, $ \tau(x)(r)$ depends only on the restriction of $ x $ to $ M_{1}\cap M_{2} $. For this, consider an element $ y\in A^{S} $ such that $ x|_{M_{1}\cap M_{2}} = y|_{M_{1}\cap M_{2}} $. Let us take an element $ z\in A^{S} $ such that $ z|_{M_{1}} = x|_{M_{1}} $ and $ z|_{M_{2}} = y|_{M_{2}} $ (we may take for instance the configuration $ z\in A^{S}$ which concide with $x$ on $M_1$ and with $y$ on $S\setminus M_1$). We have $ \tau_S(x)(r) = \tau_S(z)(r) $ since $ x $ and $ z $ coincide on $ M_{1} $, which is a memory set for $ \tau_S $. On the other hand, we have $ \tau_S(y)(r) = \tau_S(z)(r) $
since $ y $ and $ z $ coincide on $ M_{2} $, which is also a memory set for $ \tau_S $ . It follows that $ \tau_S(x)(r) = \tau_S(y)(r) $. Thus there exists a map $ \mu: A^{M_{1}\cap M_{2}}\longrightarrow A $ such that	
		\begin{center}
			$ \tau_S(x)(r) = \mu(x|_{M_{1}\cap M_{2}}) $ for all $ x\in A^{S} $.
		\end{center}
As $ \tau_S $ is $ S $-equivariant (by using Proposition 3.12.), we deduce that $ M_{1}\cap M_{2} $ is a memory set for $ \tau_S $ by using Proposition 3.13..

\end{proof}
\begin{prop}
	Let $ \tau : A^{R}\longrightarrow A^{R} $ be a cellular automaton and $S=Stab(x, Eq(\tau))$. Then there exists a unique memory set $ M_{0}\subset R $ for $ \tau_S $ of minimal cardinality. Moreover, if $ M $ is a finite subset of $ R $, then $ M $ is a memory set for $ \tau $ if and only if	$ M_{0}\subset M $.
\end{prop}

\begin{proof}
	Let $ M_{0} $ be a memory set for $ \tau_S $ of minimal cardinality. Let $ M^{\prime}_{0} $ be another memory set for $ \tau_S $ of minimal cardinality. Then by using Lemma 3.15. $ M_{0}\cap M^{\prime}_{0} $ is also a memory set for $ \tau_S $. But $ M_{0}\cap M^{\prime}_{0}\subset M_{0} $ and $ M_{0}\cap M^{\prime}_{0}\subset M_{0}^{\prime} $ which is a contradiction to the minimality of $ M_{0} $ and $ M_{0}^{\prime} $. Hence $ M_{0} $ must be equal to $ M_{0}^{\prime} $, that is, $M_{\circ}$ is a unique memory set for  $ \tau_S $ of minimal cardinality.
	
	Next, suppose that $ M $ is a finite subset of a rack $ R $ and $ M_{0}\subset M $.  Then $ M $ is also a memory set for $\tau_S $. Conversely, let $ M $ be a memory set for $ \tau_S $. As $ M_{0}\cap M $ is a memory set for $ \tau_S $ by Lemma 3.15., we have $ |M\cap M_{0}|\geq|M_{0}| $. This implies that $ M\cap M_{0} = M_{0} $, that is, $ M_{0}\subset M $. In particular, $ M_{0} $ is the unique memory set of minimal cardinality.
\end{proof}

The memory set of minimal cardinality of a cellular automaton  $ \tau_S $ is called its \textit{minimal memory set}.

\begin{rem}
Let $ \tau : A^{R}\longrightarrow A^{R} $ be a cellular automaton and $S=Stab(x, Eq(\tau))$. A map $ F : A^{R}\longrightarrow A^{R} $ is constant if there exists a configuration $x_0 \in A^R$ such that $F(x) = x_0$ for all $x_0 \in A^R$. By $S$-equivariance, a cellular automaton $ \tau_S : A^{R}\longrightarrow A^{R} $ is constant if and only if there exists $a\in A$ such that $\tau_S(r) = a$ for all $x \in A^S$ and $r\in S$. Observe that a cellular automaton
$ \tau_S : A^{R}\longrightarrow A^{R} $ is constant if and only if its minimal memory set is the empty set.
	
\end{rem}

\begin{center}
	\section{\textsc{Continuity of Cellular Automata}}	
\end{center}
\paragraph{} In this section we discuss the continuity of cellular automata on the configuration space $A^{R}$ for a rack $R$ and alphabet $A$. For this we first define the prodiscrete  topology on the configuration space $A^R$.
\paragraph{\textbf{The Prodiscrete Topology.}} Let $A^R$ be the configuration space on a rack $ R $ and the alphabet $ A $. Consider the discrete topology on each factor $ A $ of $ A^{R} $. Then the discrete topology on $ A^{R} $ can be considered as a  product topology of the discrete topologies on the factors of $A^{R}$. This topology is called the \textit{prodiscrete topology} on $ A^{R} $. This is the smallest topology on $ A^{R} $ for which the projection map $ \pi_{r} : A^{R}\longrightarrow A $, given by $ \pi_{r}(x) = x(r) $, is continuous for every $ r \in R $. The elementary cylinders
\begin{center}
	$ C(r,a) = \pi^{-1}_{r}(\{a\}) = \{x \in  A^{R} : x(r) = a\} ~~~(r \in R, a \in A) $
\end{center}
are both open and closed in $ A^{R} $. The set of all elementary cylinders $ C(r,a) $ forms a subbase for the prodiscrete topology on $ A^{R} $, that is, a subset $ U \subset A^{R} $ is open if and only if $ U $ can be expressed as a (finite or infinite) union of finite intersections of elementary cylinders $ C(r,a) $ for all $r\in R$ and $a\in A$.

For a finite subset $ \Omega\subset R $ and a configuration $ x\in A^{R}$ let $ x|_{\Omega}\in A^{\Omega} $ denote the restriction of $ x $ to $ \Omega $, that is, the map $ x|_{\Omega}:\Omega \longrightarrow A $ defined by $ x|_{\Omega}(r) = x(r) $
for all $ r \in \Omega. $ Then a neighborhood base of $ x $ is given by the sets
\[
V (x,\Omega) = \{y\in A^{R} : x|_{\Omega} = y|_{\Omega} \}=  \bigcap_{r\in  \Omega} C(r, x(r)),\]
where  $ \Omega $
runs through all finite subsets of $ R $.\\

Note that an action $\cdot$ of a rack $R$ on a topological space $X$ is said to br \textit{continuous} if the map $\varphi_r: X\longrightarrow X$ given by $\varphi_r(x)=r\cdot x$ is continuous on $X$ for all $r\in R$ and $x\in A^R$.

\begin{prop}
	Let $A^{R}$ be a configuration space on a rack $R$ and the set $A$. Then the shift action $\cdot: R\times A^R \longrightarrow A^r$, defined by $r\cdot x:=x\circ \phi^{-1}_r$, is continuous.
\end{prop}
\begin{proof}
	Let $\varphi_{r}:A^{R}\rightarrow A^{R}$ be a function defined as $ \varphi_{r}(x)=r \cdot x $ for all $r\in R$ and $x\in A^{R}$. Consider the composition map $\pi_{s}\circ \varphi_{r}:A^{R}\rightarrow A $, where $\pi_{s}$ is the continuous projection map $ \pi_{r} : A^{R}\longrightarrow A $, given by $ \pi_{s}(x) = x(s) $ for all $s\in R$. Then $\pi_{s}\circ \varphi_{r}=\pi_{r\rhd^{-1}s}$ since for all $x\in A^{R}$,
	$\pi_{s}\circ \varphi_{r}(x) =\pi_{s}(r\cdot x)=(r\cdot x)(s)=(x\circ \phi_{r}^{-1})(s)=x(\phi_{r}^{-1}(s))=x(r\rhd^{-1}s)=\pi_{r\rhd^{-1}s}(x)$. Since $\pi_{r\rhd^{-1}s}$ is continuous, $\pi_{s}\circ \varphi_{r}$ is continuous, which is possible only when $\varphi_{r}$ is continuous. Hence the action of rack $R$ on $A^{R}$ is continuous.
\end{proof}

\begin{prop}
	Let $R$ be a rack and $A$ be a set. Then every cellular automaton $\tau:A^{R}\rightarrow A^{R}$ is continuous.
\end{prop}
\begin{proof}
	Let $M\subset R$ be a memory set for cellular automaton $\tau:A^{R}\rightarrow A^{R}$. Let $x\in A^{R}$ and let $W$ be a neighbourhood of the configuration $\tau(x)$ in $A^{R}$. Then we can find a finite subset $\Omega \subset R$ such that
	\begin{equation*}
	V(\tau(x),\Omega)\subset W.
	\end{equation*}
	Consider the finite subset $ \Omega \rhd^{-1} M = \{r\rhd^{-1} m: r\in R, m\in M\} $ of $R$. Now consider a set $\tau (V(x,\Omega \rhd^{-1}M))$. Let $ \tau(z)\in \tau (V(x,\Omega \rhd^{-1}M)) $, where $z\in V(x,\Omega \rhd^{-1}M)$. Then $ z|_{\Omega \rhd^{-1}M}=x|_{\Omega \rhd^{-1}M} $. Then $ \tau(z)|_{\Omega}=\tau(x)|_{\Omega} $ by Proposition 3.0.15.
	This shows that $\tau(z)\in V(\tau(x),\Omega)$, and therefore
	\begin{equation*}
	\tau (V(x,\Omega \rhd^{-1}M))\subset V(\tau(x),\Omega)\subset W.
	\end{equation*}
	This shows that $\tau$ is continuous.
\end{proof}
\paragraph{} Next we prove an analogue of Curtis-Hedlund's Theorem for cellular automata on a rack $R$ with finite alphabet $A$. Note that for a finite alphabet $A$, the prodiscrete topology on $A^R$ is the  product of finite discrete topological spaces, and therefore, it is compact, that is, every open cover of $A^R$ has a finite subcover.
\begin{thm}
	Let $R$ be a rack and $A$ be a finite set. Let $\tau:A^{R}\rightarrow A^{R}$ be a map and $S=Stab(x, Eq(\tau))$. Equip $A^{S}$ with its prodiscrete topology. Then the following conditions are equivalent:
	\item{(a)}~~the map $\tau_S:A^{S}\rightarrow A^{S}$ is a cellular automaton;
	\item{(b)}~~the map $\tau_S$ is $S$-equivariant and continuous.
\end{thm}
\begin{proof}
	The fact that (a) implies (b) directly follows from Proposition 3.12. and Proposition 4.2..
	
	\ Conversly suppose (b). We show that $\tau_S$ is a cellular automaton. As the map $\phi_S:A^{S}\rightarrow A$, defined by $\phi_S(x)=\tau_S(x)(r)$ for all $x\in A^{R}~r\in S$, is continuous, we can find, for each $x\in A^{S}$, a finite subset $\Omega_{x} \subset S$ such that if $y\in A^S$ coincide with $x\in A^{S}$ on $\Omega_{x}$, that is, if $y\in V(x,\Omega_{x})$ such that $y|_{\Omega_{x}}=x|_{\Omega_{x}}$, then $\tau_S(y)(r)=\tau_S(x)(r)$ for all $r\in S$. The sets $V(x,\Omega_{x})$ form an open cover of $A^{S}$. As $A^{S}$ is compact, there is a finite subset $F\subset A^{S}$ such that the sets $V(x,\Omega_{x})$, $x\in F$, cover $A^{S}$. Let $M=\cup_{x\in F}\Omega_{x}$ and suppose that $y,z\in A^{S}$ such that $y|_{M}=z|_{M}$. Let $x_0\in F$ be such that $ y\in V(x_0,\Omega_{x_0}) $, that is, $y|_{\Omega_{x_0}}=x_0|_{\Omega_{x_0}}$. As $M \supset \Omega_{x_0}$ we have $y|_{\Omega_{x_0}}=z|_{\Omega_{x_0}}$ and therefore $\tau_S(y)(r)=\tau_S(t)(r)=\tau_S(z)(r)$. Thus there is map $\mu:A^{M}\rightarrow A$ such that $\tau_S(x)(r)=\mu(x|_{M})$ for all $x\in A^S$. As $\tau_S$ is $S$-equivariant, it follows from Proposition 3.13. that $\tau_S$ is cellular automaton with memory set $M$ and local defining map $\mu$.
\end{proof}

\ When the alphabet $A$ is infinite, a continuous and $S-$equivariant map $\tau_S:A^{S}\rightarrow A^{S}$ in Theorem 5.3. may fail to be a cellular automaton. In other words, the implication (b) ⇒ (a) in Theorem 15.3. becomes false if we suppress the finiteness hypothesis on $A$. This is shown by the following example.
\begin{exa}
	Let $G$ be an arbitrary infinite group and take $A=G$ as the alphabet. Consider $G$ as a conjugation rack with $\rhd$ and $\rhd^{-1}$ defined by $r\rhd s=rsr^{-1}$ and $r\rhd^{-1}s=r^{-1}sr$ for all $r,s \in G$. Consider the map $\tau:A^{G}\rightarrow A^{G}$ defined by
	\begin{center}
		$ \tau(x)(r)=x(r\rhd x(r))$
	\end{center}
	for all $x\in A^R$ and $r\in G$. Let $S=Stab(x, Eq(\tau))\subset G$. Note that $S$ is a non-empty set because $1_G\in S$. Then $\tau_S:A^{S}\rightarrow A^{S}$ defined by $ \tau(x)(r)=x(r\rhd x(r))$, for all $s \in S$, is $S$-equivariant. Next, we prove that $\tau_S$ is continuous by showing that for a given $x\in A^{S}$ and a finite subset $K\subset S$, there exists a finite set $F\subset S$ such that, if $y\in A^{S}$ and $y\in V(x, F)$, then $\tau(y)\in V(\tau(x), K)$. For this set $F=K \cup \{a\rhd x(k): k\in K\}$. Then, if $y\in V(x,F)$, then, for all $k\in K$, we have
	\begin{center}
		$ \tau_S(x)(k)=x(k\rhd x(k))=y(k\rhd x(k))=y(k\rhd y(k))=\tau(y)(k) $.
	\end{center}
	This shows that $\tau_S(y)\in V(\tau(x), K)$, and therefore, $\tau_S$ is continuous.
	\ However, $\tau_S$ is not a cellular automata. Indeed,  $r_{0}\in S\setminus \{1_{G}\}$ and, for all $r\in R$, consider the configurations $x_{r}$ and $y_{r}$ in $A^{S}$ defined by
	\begin{center}
		$	x_{r}(s)  =
		\begin{cases}
		r ~~~~~~~ \text{if}~~
		s=1_{G} \\
		r_{o} ~~~~~~\text{if}~~
		s=r\\
		1_{G}~~~~~\text{if}~~\text{otherwise}
		\end{cases}$
	\end{center}
	and
	
	\begin{center}
		$	y_{r}(s)  =
		\begin{cases}
		r~~~~~~~ \text{if}~~
		s=1_{G}\\
		1_{G}~~~~~\text{if}~~
		\text{otherwise}\\
		\end{cases}$
	\end{center}
	for all $s\in S$. Note that $x_{r}|_{S\setminus \{r\}}=y_{r}|_{S\setminus \{r\}}$. Let $F\subset S$ be a finite set and choose $r\in S\setminus F$. Then one has $x_{r}|_{F}=y_{r}|_{F}$.
	While
	\begin{center}
		$\tau(x_{r}(1_{G}))=x_{r}(x_{r}(1_{G}))=x_{r}(r)=r_{o}$
	\end{center}
	and
	\begin{center}
		$\tau(y_{r}(1_{G}))=y_{r}(y_{r}(1_{G}))=y_{r}(r)=1_{G}$.
	\end{center}
	so that $\tau(x_{r}(1_{G}))\neq \tau(y_{r}(1_{G}))$. It follows that there is no finite subset $F\subset S$ such that, for all $x\in A^{S}$, $\tau(x)(1_{G})$  depends only on the values of $x|_{F}$. This shows that $\tau|_S$ is not a cellular automaton.
\end{exa}

\begin{center}
	\section{\textsc{Composition of Cellular Automata}}	
\end{center}
\paragraph{} In this section we study the set of all cellular automata on racks with the binary operation of composition of cellular automata. We denote the composition of cellular automata on racks by right black triangle $\blacktriangleright$.

\begin{prop}
Let $R$ be a rack and $A$ be a set. Let $\sigma:A^{R}\rightarrow A^{R}$ and $\tau:A^{R}\rightarrow A^{R}$ be cellular automata. Let $S_1=Stab(x, Eq(\sigma))$, $S_2=Stab(x, Eq(\tau))$ and $S=S_1\cap S_2$. Then the composite map $\sigma_S \blacktriangleright \tau_S:A^{S}\rightarrow A^{S}$ of $\sigma_S$ and $\tau_S$ is a cellular automaton. Moreover,  if $M_1$ and $M_2$ are memory sets for $\sigma$ and $\tau$ respectively, then $M_1\rhd^{-1}M_2=\{m_1\rhd^{-1}m_2: m_1\in M_1,m_2\in M_2\}$ is a memory set for $\sigma_S \blacktriangleright \tau_S$.
\end{prop}
\begin{proof}
Since $\sigma_S$ and $\tau_S$ are $S$-equivariant, therefore for all $r\in S$ and $x\in A^{S}$, we have $(\sigma_S\blacktriangleright \tau_S)(r\cdot x)=\sigma_S(\tau_S(r\cdot x)) =\sigma_S(r\cdot \tau_S(x)) =r\cdot\sigma_S(\tau_S(x)) =r\cdot(\sigma_S \blacktriangleright \tau_S)(x)$.
This shows that $\sigma_S\blacktriangleright\tau_S$ is $S$-equivariant.

\ Now for every $x\in A^{S}$ and $r\in S$,
		\begin{center}
			 $(\sigma_S\blacktriangleright\tau_S)(x)(r)=\sigma_S(\tau_S(x))(r)$.
		\end{center}
By Remark 3.2.(c)., $\sigma_S(\tau_S(x))(r)$ depends only on the restriction of $\tau_S(x)$ to $r \rhd^{-1}M_1$. By Remark 3.2.(c). again, $\tau_S(x)(r)$ depends only on restriction of $x$ to $r\rhd^{-1}M_2$. Therefore, $(\sigma_S\blacktriangleright\tau_S)(x)(r)$ depends only on restriction of $x$ to $M_1\rhd^{-1}M_2$. Hence, by using Proposition 3.12., $\sigma_S\blacktriangleright\tau_S$ is a cellular automaton admitting $M_1\rhd^{-1}M_2$ as a memory set.
\end{proof}

\begin{rem}
With the assumptions and notations used in the Proposition 5.1., denote by $\mu_1:A^{M_1}\rightarrow A$ and $\mu_2:A^{M_2}\rightarrow A$ the local defining maps for $\sigma_S$ and $\tau_S$ respectively. Then, the local defining map $\mu_3:A^{M_1\rhd^{-1}M_2}\rightarrow A$ for $\sigma_S\blacktriangleright \tau_S$ can be described as follows.

\ For $ y\in A^{M_1\rhd^{-1}M_2} $ and $m_1\in M_1$ define $y_{m_1}\in A^{M_2}$ by setting $y_{m_1}(m_2)=y(m_1\rhd^{-1}m_2)$ for all $m_2\in M$. Also denote $\overline{y}\in A^{M_1}$ the map defined as $\overline{y}(m_1)=\mu_2(y_{m_1})$ for all $m_1\in M_1$. Now define the map $\mu:A^{M_1\rhd^{-1}M_2}\rightarrow A$ as $\mu(y)=\mu_1(\overline{y})$ for all $y\in A^{M_1\rhd^{-1}M_2}$.

\ Let $x\in A^{M_1}$, $r\in S$, $m_1\in M_1$ and $m_2\in M_2$. Then, we have
	\begin{equation*}
	\begin{split}
	(m_1\cdot (r\cdot x))|_{M_2}(m_2)=(m_1\cdot x)|_{M_2}(m_2)&=(x\circ \phi_{m_1}^{-1})(m_2)\\
	&=x(m_1\rhd^{-1}m_2)\\
	&=(r\cdot x)(m_1\rhd^{-1}m_2)\\
	&=(r\cdot x)|_{M_1\rhd^{-1}M_2}(m_1\rhd^{-1}m_2)\\
	&=((r\cdot x)|_{M_1\rhd^{-1}M_2})_{m_1}(m_2).
	\end{split}
	\end{equation*}
This shows that
	\begin{equation}\label{3}
	(m_1\cdot (r\cdot x))|_{M_2}=((r\cdot x)|_{M_1\rhd^{-1}M_2})_{m_1}
	\end{equation}
and therefore

\begin{center}
	$\tau_S(r\cdot x)(m_1)=\mu_2(m_1\cdot(r\cdot x)|_{M_2})
	=\mu_2((r\cdot x)|_{M_1\rhd^{-1}M_2})_{m_1}
	=(\overline{(r\cdot x)|_{M_1\rhd^{-1}M_2}})(m_1)$.
\end{center}
As a consequence,
	\begin{equation}\label{4}
	\tau_S(r\cdot x)|_{M_1}=(\overline{(r\cdot x)|_{M_1\rhd^{-1}M_2}})
	\end{equation}
Finally, we have
\begin{center}
$(\sigma_S\blacktriangleright \tau_S)(x)(r)=\sigma_S(\tau_S(x)(r))=\mu_1(r\cdot \tau_S(x)|_{M_1})=\mu_1(\tau_S(r\cdot x)|_{M_1})=\mu_1(\overline{(r\cdot x)|_{M_1\rhd^{-1}M_2}})=\mu(r\cdot x|_{M_1\rhd^{-1}M_2})$.
\end{center}

Hence $M_1\rhd^{-1}M_2$ is a memory set for $\sigma_S\blacktriangleright \tau_S$.
\end{rem}

\paragraph{}
Let $R$ be a rack and $A$ be a set. We denote the set of  all cellular automata $\tau:A^{R}\rightarrow A^{R}$ by $CA(R; A)$. Let $S=\bigcap\limits_{\tau\in CA(R; A)}Stab(x, Eq(\tau))$. Then the set $CA(S; A)$ consists of all cellular automata $\tau_S:A^{S}\rightarrow A^{S}$ whose memory sets are subsets of $S$.

\begin{prop}
Let $CA(R; A)$ be the set of all cellular automata $\tau:A^{R}\rightarrow A^{R}$ on a rack $R$. Let $S=\bigcap\limits_{\tau\in CA(R; A)}Stab(x, Eq(\tau))$ and $CA(S; A)$ is the of all cellular automata $\tau_S:A^{S}\rightarrow A^{S}$. Then the pair $(CA(S; A), \blacktriangleright)$ is a shelf under the composition $\blacktriangleright$ of cellular automata.
\end{prop}
\begin{proof}
By using Proposition 5.1., the set $CA(S; A)$ is closed under the composition $\blacktriangleright$ of cellular automata. We show that $\blacktriangleright$ is self-distributive on $CA(S; A)$. Let $\sigma_S$ , $\tau_S$, $\psi_S\in CA(S; A)$ with memory sets $M_{1}$, $M_{2}$ and $M_{3}$ respectively.
Let $\mu:A^{M_{1}\rhd^{-1}M_{2}}\rightarrow A$ be a local defining map for $\sigma_S\blacktriangleright \tau_S$. Also for $\sigma_S\blacktriangleright \psi_{S}$, $\nu:A^{M_{1}\rhd^{-1}M_{3}}\rightarrow A$ be a local defining map. Then, the local defining map $\kappa:A^{(M_{1}\rhd^{-1}M_{2})\rhd^{-1}(M_{1}\rhd^{-1}M_{3})}\rightarrow A$ may be described in the following way.
	
\ For $ y\in A^{(M_{1}\rhd^{-1}M_{2})\rhd^{-1}(M_{1}\rhd^{-1}M_{3})} $ and $(m_1\rhd^{-1}m_2)\in (M_{1}\rhd^{-1}M_{2})$, define $y_{m_1\rhd^{-1}m_2}\in A^{M_{1}\rhd^{-1}M_{3}}$ by;
	\begin{equation*}
	y_{m_1\rhd^{-1}m_2}(m^{\prime}_1\rhd ^{-1}\acute{m_3})=y((m_1\rhd^{-1}m_2)\rhd^{-1}(m^{\prime}_1\rhd^{-1}\acute{m_3}))
	\end{equation*}
for all $m^{\prime}_1\rhd^{-1} m_3\in M_{1}\rhd^{-1}M_{3}$. Also, denote $\overline{y}\in A^{M_{1}\rhd^{-1}M_{2}}$ the map defined by $\overline{y}(m_1\rhd^{-1}m_2)=v(y_{s\rhd^{-1}u})$ for all $m_1\rhd^{-1}m_2 \in M_{1}\rhd^{-1}M_{2}$. Now define $k(y)=\mu(\overline{y})$ for all $y\in A^{(M_{1}\rhd^{-1}M_{2})\rhd^{-1}(M_{1}\rhd^{-1}M_{3})} $. Let $x\in A^{R}$, $r\in S$, $ m_1\rhd^{-1}m_2\in M_{1}\rhd^{-1}M_{2} $ and $m^{\prime}_1\rhd^{-1}m_3\in M_{1}\rhd^{-1}M_{3}$. Then we have
	\begin{equation*}
	\begin{split}
	((m_1\rhd^{-1}m_2)\cdot(r\cdot x))|_{M_{1}\rhd^{-1}M_{3}}(m^{\prime}_1\rhd^{-1}m^{\prime}_2)&=((m_1\rhd^{-1}m_2)\cdot x)|_{M_{1}\rhd^{-1}M_{3}}(m^{\prime}_1\rhd^{-1}m^{\prime}_2)\\
	&=(x\circ \phi_{m_1\rhd^{-1}m_2}^{-1})(m^{\prime}_1\rhd^{-1}m^{\prime}_2)\\
	&=x(\phi_{m_1\rhd^{-1}m_2}^{-1}(m^{\prime}_1\rhd^{-1}m^{\prime}_2)\\
	&=x((m_1\rhd^{-1}m_2)\rhd^{-1}(m^{\prime}_1\rhd^{-1}m^{\prime}_2)\\
	&=(r \cdot x)((m_1\rhd^{-1}m_2)\rhd^{-1}(m^{\prime}_1\rhd^{-1}m^{\prime}_2).
	\end{split}
	\end{equation*}
This shows that $((m_1\rhd^{-1}m_2)\cdot(r\cdot x))|_{M_{1}\rhd^{-1}M_{3}}(m^{\prime}_1\rhd^{-1}m^{\prime}_2)=(r\cdot x)|_{(M_{1}\rhd^{-1}M_{2})\rhd^{-1}(M_{1}\rhd^{-1}M_{3})}((m_1\rhd^{-1}m_2)\rhd^{-1}(m^{\prime}_1\rhd^{-1}m^{\prime}_2)=((r\cdot x)|_{(M_{1}\rhd^{-1}M_{2})\rhd^{-1}(M_{1}\rhd^{-1}M_{3})})_{m_1\rhd^{-1}m_2}(m^{\prime}_1\rhd^{-1}m^{\prime}_2)$. That is,
	\begin{equation}\label{5}
	((m_1\rhd^{-1}m_2)\cdot (r\cdot x))|_{M_{1}\rhd^{-1}M_{3}}=((r\cdot x)|_{(M_{1}\rhd^{-1}M_{2})\rhd^{-1}(M_{1}\rhd^{-1}M_{3})})_{m_1\rhd^{-1}m_2}.
	\end{equation}
Now
	\begin{equation*}
	\begin{split}
	(\sigma_S\blacktriangleright \psi_S)(r\cdot x)(m_1\rhd^{-1}m_2)&=\nu((m_1\rhd^{-1}m_2)\cdot r\cdot x)|_{M_{1}\rhd^{-1}M_{3}} )\\~~~\text{(from~equation~\ref{5})}
	&=\nu((r\cdot x)|_{(M_{1}\rhd^{-1}M_{2})\rhd^{-1}(M_{1}\rhd^{-1}M_{3})})_{m_1\rhd^{-1}m_2}\\
	&=\overline{((r\cdot x)|_{(M_{1}\rhd^{-1}M_{2})\rhd^{-1}(M_{1}\rhd^{-1}M_{3})})}(m_1\rhd^{-1}m_2).
	\end{split}
	\end{equation*}
This shows that
	\begin{equation}\label{6}
	(\sigma_S\blacktriangleright \psi_S)(r\cdot x)=\overline{((r\cdot x)|_{(M_{1}\rhd^{-1}M_{2})\rhd^{-1}(M_{1}\rhd^{-1}M_{3})})}.
	\end{equation}
Now since ${M_{1}\rhd^{-1}(M_{2}\rhd^{-1}M_{3})\subset(M_{1}\rhd^{-1}M_{2})\rhd^{-1}(M_{1}\rhd^{-1}M_{3})}$,
	\begin{equation*}
	\begin{split}
	((\sigma_S\blacktriangleright \tau_S)\blacktriangleright (\sigma_S\blacktriangleright \psi_S))(x)(r)&=(\sigma_S\blacktriangleright \tau_S) ((\sigma_S\blacktriangleright \psi_S)(x)(r))\\
	&=\mu(r\cdot (\tau_S\blacktriangleright \psi_S)(x)|_{M_{1}\rhd^{-1}M_{3}})\\
	&=\mu((\tau_S\blacktriangleright \psi_S)(r\cdot x)|_{M_{1}\rhd^{-1}M_{3}})\\~~~\text{(from~equation~\ref{6})}
	&=\mu\overline{((r.x)|_{(M_{1}\rhd^{-1}M_{2})\rhd^{-1}(M_{1}\rhd^{-1}M_{3})})}\\
	&=\kappa((r\cdot x)|_{(M_{1}\rhd^{-1}M_{2})\rhd^{-1}(M_{1}\rhd^{-1}M_{3})})\\
	&=\kappa(r\cdot x|_{M_{_{1}}\rhd^{-1}(M_{2}\rhd^{-1}M_{3})})
	\end{split}
	\end{equation*}
This shows that
	\begin{equation*}
	((\sigma_S\blacktriangleright \tau_S)\blacktriangleright (\sigma_S\blacktriangleright \psi_S))(x)(r)=\kappa(r\cdot x|_{M_{1}\rhd^{-1}(M_{2}\rhd^{-1}M_{3})})=	\sigma_S\blacktriangleright (\tau_S)\blacktriangleright  \psi_S)(x)(r).
	\end{equation*}
for all $x\in A^S$ and $r\in R$. Hence the set $CA(S,A)$ is a shelf under the composition of cellular automata.
	
\end{proof}

\begin{prop}
Let $CA(R; A)$ be the set of all cellular automata $\tau:A^{R}\rightarrow A^{R}$ on a quandle $R$. Let $S=\bigcap\limits_{\tau\in CA(R; A)}Stab(x, Eq(\tau))$ and $(CA(S; A), \blacktriangleright)$ is a shelf of all cellular automata $\tau_S:A^{S}\rightarrow A^{S}$. Then $\blacktriangleright$ is idempotent in $CA(S; A)$.
\end{prop}
\begin{proof}
Let  $\tau_S\in CA(S,A)$ with memory set $M$. Since $R$ is a quandle $M\rhd^{-1}M=M$. Then by using Proposition 5.1., $ \tau_S\blacktriangleright \tau_S $ is a cellular automaton with memory set $M\rhd^{-1}M=M$. Theorefor, $\tau_S\blacktriangleright \tau_S = \tau_S$.
\end{proof}

\begin{rem}
In order to make the shelf $CA(S; A), \blacktriangleright)$ a rack, one needs another binary operation $\blacktriangleright^{-1}$ on $CA(S; A)$ such that for all $ \sigma_S\in CA(S,A) $ there exists $\tau_S\in CA(S,A) $ and
\begin{center}
	$\sigma_S\blacktriangleright (\sigma_S\blacktriangleright^{-1}\tau_S)=\tau_S= \sigma_S\blacktriangleright^{-1} (\sigma_S\blacktriangleright\tau_S)$.
\end{center}
In particular, for a trivial rack $R$, $(CA(R,A),\blacktriangleright,\blacktriangleright^{-1})$ is a rack with $\sigma\blacktriangleright\tau=\sigma\blacktriangleright^{-1}\tau=\tau$ for all $\sigma, \tau \in CA(R,A)$.

\end{rem}

\paragraph{\textbf{Invertible Cellular Automata.}}
Let $R$ be a rack and let $A$ be a set. Then a cellular automaton
$\tau : A^R \longrightarrow A^R$ is invertible (or reversible) if $\tau$ is bijective and the inverse
map $\tau^{-1} : A^R \longrightarrow A^R$ is also a cellular automaton. This is equivalent to the
existence of a cellular automaton $\sigma : A^R \longrightarrow A^R$ such that $\sigma\blacktriangleright\tau=\tau\blacktriangleright \sigma=Id_{A^R}$, where $Id_{A^R}$ is the identity map on $A^R$. Note that for a group $G$ the identity map $Id_{A^G}$ is a cellular automaton with the memory set $\{1_G\}$. However, since the identity element does not exist in racks, the identity map $Id_{A^R}$ on $A^R$ may not be a cellular automaton.

\ Every bijective cellular automaton on racks is not always invertible. However, certain equivariant bijective cellular automata on racks are always invertible by the following theorem.

\begin{thm}
Let $R$ be a rack and let $A$ be a finite set. Let $CA(R; A)$ be the set of all cellular automata $\tau:A^{R}\rightarrow A^{R}$, and $S=\bigcap\limits_{\tau\in CA(R; A)}Stab(x, Eq(\tau))$. Then every bijective cellular automaton $\tau_S:A^{S}\rightarrow A^{S}$ is invertible.	
\end{thm}

\begin{proof}
Let $\tau_S:A^{S}\rightarrow A^{S}$ be a bijective cellular automaton. The map $\tau^{-1}_{S}$ is $S$-equivariant since $\tau_S$ is $S$-equivariant. On the other hand, $\tau^{-1}_{S}$ is continuous with respect to the prodiscrete topology by compactness of $A^R$. Consequently, $\tau^{-1}_{S}$ is a cellular automaton by Theorem 4.3.
	
\end{proof}

\begin{center}
	\section{\textsc{Conclusions}}
\end{center}
There are open problems in the study of cellular automata on racks which are analogous to the classical theorems on cellular automata on groups. For instance, one can look for the Garden of Eden theorem for cellular automata on racks which may characterize the surjective cellular automata on racks as pre-injective cellular automata. One can also look for the existence of a non-equivariant, non-invertible bijective cellular automaton on racks?

\newpage

\begin{center}

\end{center}
Naqeeb ur Rehman (Corresponding author), Allama Iqbal Open University Islamabad, Pakistan.\\
Email: naqeeb@aiou.edu.pk\\
Muhammad Khuram Shahzad, Allama Iqbal Open University Islamabad, Pakistan.\\
Email: aabir25121986@gmail.com

\end{document}